\documentclass[]{amsart}
\usepackage{amsmath}
\usepackage{amsfonts}
\usepackage{amssymb}
\usepackage{amsbsy}
\usepackage{amscd}
\usepackage{amsthm}
\usepackage{enumitem}
\usepackage{color}
\usepackage{framed}
\usepackage{graphicx}

\newtheorem{theorem}{Theorem}[section]
\newtheorem{lemma}[theorem]{Lemma}
\newtheorem{corollary}[theorem]{Corollary}

\newtheorem{proposition}[theorem]{Proposition}

\theoremstyle{remark}
\newtheorem{remark}[theorem]{Remark}

\numberwithin{equation}{section}

\newcommand{\T}{\mathbb{T}}
\newcommand{\cT}{\mathcal T}
\newcommand{\cP}{\mathcal P}
\newcommand{\R}{\mathbb{R}}

\newcommand{\N}{\mathbb{N}}

\newcommand{\cO}{\mathcal O}
\newcommand{\cU}{\mathcal U}
\newcommand{\eps}{\varepsilon}

\renewcommand{\phi}{\varphi}

\begin{document}

\title[Smooth rigidity for 3-dimensional volume preserving Anosov flows]{Smooth rigidity for 3-dimensional volume preserving Anosov flows and weighted marked length spectrum rigidity}

\author {Andrey Gogolev and Federico Rodriguez Hertz}\thanks{The authors were partially supported by NSF grants DMS-1955564 and DMS-1900778, respectively}

 \address{Department of Mathematics, The Ohio State University,  Columbus, OH 43210, USA}
\email{gogolyev.1@osu.edu}

\address{Department of Mathematics, The Pennsylvania State University,  University Park, PA 16802, USA}
\email{hertz@math.psu.edu}

\begin{abstract} 
  \begin{sloppypar}
Let $X_1^t$ and $X_2^t$ be volume preserving Anosov flows on a 3-dimensional manifold $M$. We prove that if $X_1^t$ and $X_2^t$ are $C^0$ conjugate then the conjugacy is, in fact, smooth, unless $M$ is a mapping torus of an Anosov automorphism of $\mathbb T^2$ and both flows are constant roof suspension flows. We deduce several applications. Among them is a new result on rigidity of Anosov diffeomorphisms on $\mathbb T^2$ and a new ``weighted'' marked length spectrum rigidity result for surfaces of negative curvature.
  \end{sloppypar}
\end{abstract}
\maketitle

\maketitle

\section{Introduction}

Dynamics and geometry are different fields with distinct agendas. These fields frequently enrich each other, in particular, via the interplay between hyperbolic dynamics and  negative curvature as pioneered by Hopf and Anosov. Much less frequently a problem in one field could uncover serious lack of understanding and technical weakness in the other field. This is precisely the story of this paper. 

\medskip

\subsection{Definitions}

Let $M$ be a closed smooth Riemannian manifold. Recall that a diffeomorphism $f\colon M\to M$ is called {\it Anosov} if the tangent bundle admits a $Df$-invariant splitting $TM=E^s\oplus E^u$, where  $E^s$ is uniformly contracting and $E^u$ is uniformly expanding under $f$.  It is well known that all 2-dimensional Anosov diffeomorphisms are conjugate to Anosov automorphisms of $\mathbb T^2$~\cite{Fr}.

Similarly,
a smooth flow $X^t\colon M\to M$ is called {\it Anosov} if the tangent bundle admits a $DX^t$-invariant splitting $TM=E^s\oplus X\oplus E^u$, where $X$ is the generator of $X^t$, $E^s$ is uniformly contracting and $E^u$ is uniformly expanding under $DX^t$, $t>0$. Basic examples of 3-dimensional Anosov flows are geodesic flows on surfaces of negative curvature and suspension flows of Anosov diffeomorphisms of the 2-torus. Many more 3-dimensional examples can be constructed by various surgery techniques (see~\cite{Bar} for an overview).

Conjugacy and orbit equivalence are natural equivalence relations for dynamical systems. In particular, Anosov proved that $C^1$ close Anosov diffeomorphisms are conjugate and $C^1$ close Anosov flows are orbit equivalent~\cite{An}. (Recall that flows $X_1^t$ and $X_2^t$ are {\it orbit equivalent} if there exists a homeomorphism $H\colon M\to M$ which sends orbits of $X_1^t$ to the orbits of $X_2^t$ preserving the time direction.) Anosov also observed that, in general, the continuous conjugacy or the orbit equivalence is not $C^1$ or even Lipschitz.

\subsection{Structural stability and obstructions to smooth conjugacy}

In the case of diffeomorphisms the basic obstructions to $C^1$ regularity of the conjugacy are given by eigenvalues at periodic points. The natural question is whether matching of these obstructions could guarantee $C^1$ or higher regularity of the conjugacy. One can ask for as much as coincidence of Jordan normal forms or as little as coincidence of full Jacobians at corresponding periodic points.

In the case of flows, first there are obstructions to orbit equivalence being a conjugacy of flows $H$ (that is, $H\circ X_1^t=H\circ X_2^t$ for all $t$). These obstructions are given by the periods of corresponding periodic orbits. If the flows are transitive and these obstructions vanish then the orbit equivalence can be improved to a conjugacy by adjusting this orbit equivalence along the orbits of the flows~~\cite[Theorem 19.2.9]{KH}. Once one has a true $C^0$ conjugacy $H$, the obstructions to $C^1$ regularity are, similarly to the diffeomorphisms case, given by the eigenvalues of the linearized Poincar\'e return maps at the corresponding periodic points.

\subsection{Smooth rigidity in low dimension: prior results} The main question of the smooth rigidity program in rank one is whether matching of obstructions implies regularity of the conjugacy. In low dimensions, that is dimension 2 for diffeomorphisms and dimension 3 for flows, this question was extensively studied by de la Llave, Marco, Moriy\'on and by Pollicott. Specifically, if two Anosov diffeomorphisms on $\mathbb T^2$ are conjugate and both the stable and unstable eigenvalues at corresponding periodic points are equal (vanishing of obstructions) then the conjugacy is smooth~\cite{dlL0, MM, dlL, Pol}. Similarly, if two Anosov flows are conjugate (equivalently, the periods of corresponding periodic points are equal) and the stable and unstable eigenvalues at corresponding periodic points are equal then the conjugacy is smooth~\cite{LM, dlL, Pol}.

 In fact, if diffeomorphisms (flows) have finite regularity $C^r$ (that is, $\lfloor r\rfloor$ times continuously differentiable and its $C^{\lfloor r\rfloor}$-differential is H\"older continuous with exponent $r-\lfloor r\rfloor$) then the conjugacy is $C^{r_*}$, where 
$r_*=r$ if $r\notin\N$ and $r_*=(r-1)+Lip$ if $r\in\N$. Such sharp regularity was achieved through the use of an analytic lemma, which was established on demand by Journ\'e~\cite{J} for these purposes.

\subsection{Smooth rigidity in low dimension: new results}

Our main result is the following improved smooth rigidity for flows.
\begin{theorem}
\label{thm2}
 Let $X_i^t\colon M\to M$, $i=1,2$, be $C^r$, $r>2$, volume preserving Anosov flows which are conjugate via a  conjugacy $H$. Then at least one of the following conclusions holds:
\begin{enumerate}
\item the conjugacy $H$ is a $C^{r_*}$ diffeomorphism; 
\item flows $X_i$ are constant roof suspensions of Anosov diffeomorphisms $f_i\colon\T^2\to\T^2$, $i=1,2$.
\end{enumerate}
\end{theorem}

In other words, we improve the smooth rigidity result for flows by showing that one can, in fact, discard the assumption of matching of the second collection of obstructions (eigenvalues) and conclude existence {\it and} smoothness of the conjugacy only using the first set of obstructions (periods) in all cases, but the constant roof suspension case. In the latter case, vanishing of eigenvalue obstructions is, clearly, a necessary assumption. If both flows are assumed to be contact then such rigidity was established by Feldman and Ornstein~\cite{FO}. The contact property ensures that the strong stable and unstable foliations are $C^1$ which then can be used to obtain $C^1$ regularity property of the conjugacy via the inverse function theorem. However, aside from contact flows and constant roof suspensions, the strong foliations of a volume preserving Anosov flow are merely H\"older.

We proceed with some applications.

\begin{corollary}
\label{thm1}
Let $f_i\colon\T^2\to\T^2$, $i=1,2$, be $C^r$, $r>2$, area preserving Anosov diffeomorphisms which are conjugate via a conjugacy $h$, $h\circ f_1=f_2\circ h$ and let $\phi_i\colon\T^2\to \R$ be $C^r$ smooth functions, $i=1,2$. Assume that for every periodic point $p=f_1^k(p)$ the following sums agree
$$
\sum_{i=0}^{k-1}\phi_1(f_1^i(p))=\sum_{i=0}^{k-1}\phi_2(f_2^i(h(p))
$$
Then at least one of the following holds:
\begin{enumerate}
\item conjugacy $h$ is a $C^{r_*}$ diffeomorphism;
\item the functions $\phi_i$ are cohomologous to a constant, $\phi_i=u_i\circ f_i-u_i+C$, $u_i\in C^{r_*}(\T^2)$.
\end{enumerate}
\end{corollary}

\begin{proof}
Indeed, to see that Theorem~\ref{thm2} implies the above corollary, pick a constant $c$ such that $\phi_i+c>0$, $i=1,2$. Then let $X_i^t$ be the suspension flow over $f_i$ with the roof function $\phi_i+c$. Using the assumption on the periodic orbits we have that $\phi_1+c$ is cohomologous to $\phi_2\circ h+c$ by the Livshits Theorem~\cite{Liv}. This is equivalent to $X_1^t$ being conjugate to $X_2^t$ and, hence, Theorem~\ref{thm2} applies and yields the posited dichotomy.
\end{proof}

Comparing this result to the classical rigidity results for Anosov diffeomorphisms on $\T^2$ reviewed in Section~1.2, we see that matching of abstract data given by functions $\phi_1$ and $\phi_2$ works as well as matching of natural data given by stable and unstable Jacobians of Anosov diffeomorphisms.

We also note that the above corollary can be interpreted as follows. Consider partially hyperbolic skew-products $F_i(x,y)=(f_i(x), y+\phi_i(x))$ on $\mathbb T^3$. If the skew-products are conjugate and $F_1$ is not conjugate to a product diffeomorphism (or equivalently, $\phi_1$ is not cohomologous to a constant) then the conjugacy is smooth.

\subsection{Khalil-Lafont conjecture and weighted marked length spectrum rigidity}
We proceed to a question posed by Osama Khalil and Jean Lafont. Let $S$ be a surface of genus $\ge 2$. Given a Riemannian metric of negative curvature $g$ on $S$, in each non-trivial free homotopy class of maps $\gamma\colon S^1\to S$ there exists a unique unit speed geodesic $\gamma(g)$ whose length we will denote by $\ell(\gamma,g)$.

If $g_1$ and $g_2$ are two metrics of negative curvature such that for all free homotopy classes $\gamma$ we have $\ell(\gamma, g_1)=\ell(\gamma, g_2)$ then $g_1$ and $g_2$ are isometric. This result is known as {\it marked length spectrum rigidity} for surfaces and is due to Otal and Croke~\cite{otal, croke} (independently). Khalil and Lafont suggested a generalized ``weighted'' version of  marked length spectrum rigidity. Namely, instead of assuming matching of lengths one assumes that certain weight functions $\phi_1\colon T^1S\to\R$ and $\phi_2\colon T^1S\to\R$ match along corresponding geodesics:
\begin{equation}
\label{eq_weights}
\int_{\gamma(g_1)}\phi_1(\dot\gamma(g_1)(t))dt=\int_{\gamma(g_2)}\phi_2(\dot\gamma(g_2)(t))dt
\end{equation}
What can be said about the metrics?

In order to answer this question we first will need to establish a ``weighted" version of Theorem~\ref{thm2}. Recall the following definition. A function $\phi\colon M\to \R$ is called an {\it abelian coboundary} over an Anosov flow $X^t\colon M\to M$ generated by the vector field $X$ if 
$$\phi(x)=\omega(X(x))
$$
for some closed 1-form $\omega\colon M\to T^*M$.  

\begin{theorem}
\label{thm3}
 Let $X_i^t\colon M\to M$, $i=1,2$, be $C^r$, $r>2$, contact Anosov flows which are orbit equivalent via an orbit equivalence $H$ which is $C^r$ along the orbits. Assume that $\phi_i\colon M\to\R$, $i=1, 2$, are $C^r$, $r>2$, functions such that
 $$
 \int_\beta \phi_1(\beta(t))dt=\int_{H_*\beta}\phi_2(H_*\beta(t))dt
 $$
 for every periodic orbit $\beta$ of $X_1^t$ and corresponding periodic orbit $H_*\beta$ for $X_2^t$.
 Then either $\phi_i$ is an abelian coboundary over $X_i^t$, $i=1,2$, or $H$ is $C^{r_*}$, that is, $X_1^t$ is $C^{r_*}$-smoothly orbit equivalent to $X_2^t$. 
 \end{theorem}
 
 \begin{remark}
 It is well known that any $C^0$ orbit equivalence can always be adjusted along the orbits to be as smooth as the flows along the orbits.
 \end{remark}
\begin{remark} If we additionally assume that $\phi_1$ and $\phi_2$ are positive functions then Theorem~\ref{thm3} follows rather easily from Theorem~\ref{thm2}. Indeed, in this case one can consider reparametrizations $Y_i^t$ given by scaled generators $Y_i=\frac{1}{\phi_i}X_i$. The matching condition of Theorem~\ref{thm3} then translates into matching of periods of $Y_1^t$ and $Y_2^t$. Hence we have that $Y_1^t$ is conjugate to $Y_2^t$ and Theorem~\ref{thm3} can be applied. Also note that in this case the contact assumption is only needed to rule out the constant roof suspension case.
\end{remark}

It is easy to deduce a more general version for flows which merely admit contact reparametrizations, which we state next as a corollary. The authors also plan to generalize to Theorem~\ref{thm3} to the setting of volume preserving flows in the future work.

\begin{corollary}
\label{cor5}
 Let $Y_i^t\colon M\to M$, $i=1,2$, be $C^r$, $r>2$, Anosov flows which are orbit equivalent via an orbit equivalence $H$. Also assume that both $Y_i^t$ admit contact reparametrizations. Assume that $\phi_i\colon M\to\R$, $i=1, 2$, are $C^r$, $r>2$, functions such that
 $$
 \int_\beta \phi_1(\beta(t))dt=\int_{H_*\beta}\phi_2(H_*\beta(t))dt
 $$
 for every periodic orbit $\beta$ of $Y_1^t$ and corresponding periodic orbit $H_*\beta$ for $X_2^t$.
 Then either $\phi_i$ is an abelian coboundary over $Y_i^t$, $i=1,2$, or $Y_1^t$ is $C^{r_*}$-smoothly orbit equivalent to $Y_2^t$.
\end{corollary}

Theorem~\ref{thm3} yields a solution of the Khalil-Lafont conjecture.

\begin{corollary} 
\label{cor3}
Let $S$ be a closed surface and let $g_1$ and $g_2$ be metrics of negative curvature on $S$. Let $\phi_1\colon T^1S\to\R$ and $\phi_2\colon T^1S\to\R$ be $C^r$, $r>2$, functions satisfying the matching condition~\eqref{eq_weights}. Also assume that $\phi_1$ is not an abelian coboundary over the geodesic flow of $g_1$. Then $g_1$ is homothetic to $g_2$, that is, there exists a positive constant $c$ such that $g_2$ is isometric to $c^2g_1$.
\end{corollary}
Using recent progress on marked length spectrum rigidity~\cite{GLP} the assumption on $g_1$ and $g_2$ can be relaxed to merely having Anosov geodesic flows.

In the case when the functions $\phi_i$ depends only on the base-point of the tangent vector then the Corollary~\ref{cor3} takes the following particularly nice form, since in this case the only abelian coboundary is the zero function. We also obtain matching of the weights in this case.

\begin{corollary} 
\label{cor4}
Let $S$ be a closed surface and let $g_1$ and $g_2$ be metrics of negative curvature on $S$. Let $\psi_1\colon S\to\R$ and $\psi_2\colon S\to\R$ be  non-zero functions satisfying the matching condition
\begin{equation*}
\label{eq_weights}
\int_{\gamma(g_1)}\psi_1(\gamma(g_1)(t))dt=\int_{\gamma(g_2)}\psi_2(\gamma(g_2)(t))dt
\end{equation*}
for all $\gamma$.
 Then there exists a constant $c>0$ and an isometry $f\colon(S,c^2g_1)\to (S, g_2)$. We also have matching of the weights $\psi_2\circ f=c\psi_1$.
\end{corollary}

Of course, it is very interesting and challenging to generalize these corollaries to higher dimensional setting (at least to in the cases when marked length spectrum rigidity is known). We also would like to ask if corollaries admit a generalization to the setting on non-positively curved metrics or metrics with no conjugate points. More specifically, does a  weighted version Croke-Fathi-Feldman~\cite{CFF} marked length spectrum rigidity holds? (The matching condition is imposed on the infimum of integrals over all geodesic minimizers in a given free homotopy class.) Note that in this setting the dynamical tools become much less powerful.



\subsection{Remarks}

\begin{enumerate}
\item We would like to point out that the proof of our main result crucially relies on the earlier rigidity theorems, in particular, on de la Llave-Marco-Moriy\'on and Pollicott theorem and on Feldman-Ornstein theorem. Ultimately, our proof splits into several cases. In one case we conclude that both flows are contact and we finish by citing Feldman-Ornstein theorem~\cite{FO}. In the other case we are able to recover stable and unstable eigenvalues at all periodic points from the periods of certain approximating periodic orbits. This then enables us to apply de la Llave-Marco-Moriy\'on and Pollicott theorem.
Hence our proof builds upon earlier works on rigidity and in no way discards it.
\item The proof of the main result also relies on work of Foulon and Hasselblatt~\cite{FH} on longitudinal Anosov cocycle and on very recent work of Dilsavor and Marshall Reber~\cite{DMR} on positive proportion Livshits theorem.
\item The conclusion of Theorem~\ref{thm3} (and Corollary~\ref{cor3}) is sharp in the following sense. If $\phi_1$ is an abelian coboundary then, generally speaking, one does not have a smooth orbit equivalence of the flows. Indeed, start with orbit equivalent flows $X_1^t$ and $X_2^t$ via $H$, which are not $C^1$ orbit equivalent (this is always the case when at least one pair of multipliers at corresponding periodic orbits are different). Take any closed 1-form $\omega_1$ and let $\phi_1=\omega_1(X_1)$. Let $\omega_2$ be any closed 1-form which represents the cohomology class $H_*[\omega_1]$. Let $\phi_2=\omega_2(X_2)$. Then for any peridic orbit $\beta$ of $X_1$ we have
\begin{multline*}
\,\,\,\,\,\,\,\,\,\,\,\,\int_\beta\phi_1(X_1\beta(t))dt=\langle[\omega_1],[\beta]\rangle=
\langle[H_*\omega_1],[H_*\beta]\rangle\\= \langle[\omega_2],[H_*\beta]\rangle=\int_{H_*\beta}\phi_2(X_2(H_*\beta(t))dt,
\end{multline*}
where we have used the fact that homology-cohomology pairing is independent of the choice of representatives and also its functoriality property. Hence, we have a matching pair $(\phi_1,\phi_2)$ of abelian coboundaries without having a smooth orbit equivalence.
\item Recently the current authors have written a series of papers on rigidity in rank one dynamics for expanding maps, Anosov diffeomorphisms and Anosov flows under various additional assumptions~\cite{GRH, GRH21a, GRH21b, GRH0}. All these papers utilize what we call ``matching functions technique.'' The matching functions technique seems to be quite hopeless in the setting of 3-dimensional Anosov flows. So, while the statements of results in this series of papers are very similar to our main result here (rigidity), technologically this paper is very different from our previous papers on the subject of rigidity in rank one hyperbolic dynamics.
\end{enumerate}

{ \bfseries Acknowledgements.} We would like to thank Jean Lafont and Osama Khalil for sharing their question on weighted marked length spectrum rigidity with us. It had largely motivated our interest in improving rigidity results for 3-dimensional Anosov flows. We are very grateful to Livio Flaminio for his interest in this work and enlightening discussions. We thank Martin Leguil for discussions and, especially, for pointing us to the formulas which connect periods of closed orbits and Lyapunov exponents. We also thank James Marshall Reber for checking various parts of the proof and his feedback on our drafts. Last, but not the least, we would like to thank Caleb Dilsavor and James Marshall Reber for their recent proof of Positive proportion Livshits Theorem which we needed for this paper.
\section{Preliminaries}

\subsection{De la Llave-Marco-Moriy\'on and Pollicott theorem for 3-dimensional Anosov flows}
We briefly recall the scheme of the proof of the rigidity theorem stated in Section~1.2. We make an additional assumption that the conjugate flows are both volume preserving which makes the argument more succinct. We follow~\cite{dlL}.  

The first step is to use thermodynamic formalism and the Livshits theorem to show that the conjugacy $H$ sends the invariant volume of $X_1^t$ to the invariant volume of $X_2^t$. Then one concludes that the conjugacy sends 1-dimensional conditional measures of the volume on local stable and unstable leaves of $X_1^t$ to corresponding conditional measures on local stable and unstable leaves of $X_2^t$. One can argue that these measures are smooth which immediately yields smoothness of $H$ along the stable and unstable leaves. The last step in the proof is to apply Journ\'e's regularity lemma~\cite{J} which establishes smoothness of $H$ from the smoothness along the foliations.

\subsection{Moser normal form}
We recall the classical Moser normal form for a conservative hyperbolic fixed point in dimension 2~\cite{M}. Assume that $F$ is a smooth area preserving, orientation preserving local map defined on a neighborhood of the origin and such that the origin is a hyperbolic saddle point. Then there exists a smooth area preserving change of coordinates $\Psi$ such that $F$ takes the following form
$$
\Psi^{-1}\circ F\circ \Psi(x,y)=(\mu x(1+axy+O((xy)^2)), \mu^{-1}y(1-axy+O((xy)^2)),
$$
where $\mu\in(0,1)$.
For the above formula to hold $F$ has to be at least $C^5$ or better. 

Since we will be working in $C^r$ regularity with $r>2$ only, we will need to have a weak version of the Moser normal form which holds in such low regularity. Hence, let $F$ be as before, but now only assumed to be $C^r$ regular. Let $\theta=\min\{1,r-2\}>0$. Then there exists a $C^r$ change of coordinates $\Psi$ such that
\begin{equation}
\label{eq_moser}
\Psi^{-1}\circ F\circ \Psi(x,y)=(\mu x+xy\phi_1(x,y), \mu^{-1}y+xy\phi_2(x,y) ),
\end{equation}
where $\phi_1$ and $\phi_2$ are $C^\theta$, that is, they are H\"older with exponent $\theta$ and they vanish at the origin, $\phi_1(0,0)=\phi_2(0,0)=0$. 

Existence of such a normal form is an exercise and can be established in four steps which we proceed to outline. First, by a linear change of coordinates, $F$ can be brought to the form
$$
(x,y)\mapsto(\mu x+Q_1(x,y), \mu^{-1}y+Q_2(x,y))
$$
where $Q_1$ and $Q_2$ vanish to the first order at the origin. Then using a  change of coordinates of the form $(x,y)\mapsto (x+P_1(x,y), y+P_2(x,y))$ with $P_1$ and $P_2$ being homogenous degree 2 polynomials we can ensure that $Q_1$ and $Q_2$ vanish to the second order at the origin (that is, all second order partial derivatives vanish at $(0,0)$). The third step is to ``straighten'' the stable and unstable manifolds of the saddle. We can push the stable manifold to the $x$-axis along the vertical direction and then push the unstable manifold to the $y$-axis along the horizontal direction. Notice that this change of coordinates is $C^r$ since the stable and unstable manifolds are $C^r$. In this way we bring $F$ to the form
$$
(x,y)\mapsto(\mu x+\hat Q_1(x,y), \mu^{-1}y+\hat Q_2(x,y))
$$
Since $x$ axis is now invariant, we have that $\hat Q_2=y\bar Q_2$ and, since, $y$ axis is invariant, $\hat Q_1=x\bar Q_1$.
Also notice that, since the stable and unstable manifolds were tangent to the axes to the second order, after the last coordinate change the functions $\hat Q_1$ and $\hat Q_2$ still vanish to the second order at $(0,0)$. 

Now the restriction of $F$ to the $x$-axis has the form $x\mapsto \mu x+x\bar Q_1(x,0)$. Using Poincar\'e linearization we can $C^r$ conjugate it to the linear map $x\mapsto \mu x$. It is easy to $C^r$ extend this conjugacy to the neighborhood without destroying any of the established properties. The restriction to the $y$-axis can be linearized in the same way. After this last change of coordinates the map takes the form 
$$
(x,y)\mapsto(\mu x+x\tilde Q_1(x,y), \mu^{-1}y+y \tilde Q_2(x,y))
$$
where $\tilde Q_1$ and $\tilde Q_2$ are $C^{r-1}$ and $\tilde Q_1(x,0)=\tilde Q_2(0,y)=0$ since we have linearized along the axis. Hence, we have $\tilde Q_1=y\phi_1$ and $\tilde Q_2=x\phi_2$, which yields the posited normal form~\eqref{eq_moser}. Note that it is clear that $\phi_1$ and $\phi_2$ are $C^\theta$ and they vanish at the origin since the non-linear component vanishes to the second order at $(0,0)$.

\subsection{Adapted transverse coordinates}
\label{sec23}

Now we recall the definition of adapted transverse coordinates for a $3$-dimensional volume preserving Anosov flow $X^t\colon M\to M$.
These coordinates constitute a non-stationary version of Moser normal form and were introduced by Hurder and Katok~\cite{HK}.

Assume that   $X^t\colon M\to M$ is a $C^{r}$, $r> 2$, 3-dimensional Anosov flow which preserves a volume form $\omega$. A map 
$$
\Psi\colon M\times (-\eps,\eps)^2\to M;\,\, (p,x,y)\mapsto \Psi_p(x,y)
$$
is called a {\it $C^{r}$ adapted transverse coordinate system} for $X^t$ if the following properties hold.
\begin{enumerate}
\item $\Psi_p(0,0)=p$ and the map $\Psi_p\colon (-\eps,\eps)^2\to M$ is a $C^{r}$ embedding whose image $\cT_p$ is transverse to the flow generator $X$;
\item The family of maps $\Psi\colon M\to Emb((-\eps,\eps)^2, M)$ is H\"older continuous map into the space of embeddings equipped with $C^{r}$ topology;\footnote{In fact, one can also require that $\Psi$ is $C^1$ in the first coordiate $p$ as Hurder-Katok do, but this will not be important.}
\item The ``horizontal'' vector field $(\Psi_p)_*(\frac{\partial}{\partial x})$ and ``vertical''  vector field $(\Psi_p)_*(\frac{\partial}{\partial y})$ on $\cT_p$ are transverse to the weak stable and weak unstable distributions, respectively. Further, $(\Psi_p)_*(\frac{\partial}{\partial x})$ is $C^1$ tangent to the vector field $E^{0u}\cap T\cT_p$ at $(\Psi_p)_*(\frac{\partial}{\partial x})(0,0)=E^{0u}(p)\cap T_p\cT_p$; and $(\Psi_p)_*(\frac{\partial}{\partial y})$ is $C^1$ tangent to the vector field $E^{0s}\cap T\cT_p$ at $(\Psi_p)_*(\frac{\partial}{\partial y})(0,0)=E^{0s}(p)\cap T_p\cT_p$;
\item The curve $\Psi_p((\eps,\eps), \{0\})$ is contained in the weak stable submanifold of $p$ and the curve $\Psi_p((\{0\}, (\eps,\eps))$ is contained in the weak unstable submanifold of $p$;
\end{enumerate}

Hurder and Katok proved that $C^{r-1}$ adapted transverse coordinates exist for $C^{r}$ flows and also proved that they provide a normal form for the flow~\cite[Section~4]{HK}. The reason why they lose a derivative is that they insist on an additional volume preservation property which we don't need in this paper. Namely, they require that 
$$
\Psi^*_p(\omega_{\cT_p})=dx\wedge dy, 
$$
where $\omega_{\cT_p}=\iota_X\omega$ the induced volume on $\cT_p$. To guarantee this property an additional coordinate change must made which is responsible for the loss of a derivative. Since we don't need such a property, our adapted charts are $C^r$ and provide the following normal form for the Poincar\'e return maps $F_p\colon\cT_p\to\cT_{X^{t_0}(p)}$, $t_0>0$,
\begin{equation}
\label{eq_psi}
\Psi^{-1}_{X^{t_0}(p)}\circ F_p\circ \Psi_p(x,y)=(\mu(t_0) x+o(x^2+y^2), \mu(t_0)^{-1}y+o(x^2+y^2) ),\,\,\, \mu\in(0,1)
\end{equation}
%

\begin{remark} Note that an adapted coordinate systems for $X^t$ is also an adapted coordinate system for any reparametrization of $X^t$.
\end{remark}

\subsection{Longitudinal Anosov cocycle and Foulon-Hasselblatt theorem}
\label{sec_cocycle}

Given a point $p\in M$ and a time $t\in \R$ consider the first return time $\xi$ from $\cT_p$ to $\cT_{X^t(p)}$, which is a $C^r$ function defined by two conditions
$$
\xi(p)=t,\,\,\, X^{\xi(q)}(q)\in \cT_{X^t(p)}, q\in \cT_p
$$
These conditions uniquely define $\xi$ in a neighborhood of $p$ in $\cT_p$. Recall that $\cT_p$ is equipped with adapted coordinates $(x,y)$, hence, we can view $\xi$ as a function of variables $x$ and $y$. Define {\it longitudinal Anosov cocycle} $K\colon M\times \R\to \R$ as the mixed partial derivative
\begin{equation}
\label{eq_AC}
K(p,t)=\frac{\partial^2\xi}{\partial x\partial y}(0,0)
\end{equation}
Using linearity of the partial derivative and that $\det DF_p(p)=1$ by~\eqref{eq_psi}, it is easy to check that $K$ is an additive cocycle over $X^t$. Since $\Psi_p$ is $C^r$ with $r>2$, it is immediate from property~2 of the adapted coordinates that $K$ is H\"older continuous. {\it Foulon and Hasselblatt~\cite{FH} (see also~\cite{FFH}) proved that if $K$ is a coboundary then then $E^s\oplus E^u$ is at least $C^1$, which then implies by work of Plante~\cite[Theorem~4.7]{Pl} that either $X^t$ is a contact flow or $E^s\oplus E^u$ is an integrable distribution which, in turn, implies that $X^t$ is a constant roof suspension.}

We have to remark that our definition of longitudinal Anosov cocycle is not exactly the same as the one given by Foulon-Hasselblatt. This is because they define the cocycle relative to a different collection of adapted transverse coordinates. Namely, they use transversals which contain local stable and unstable manifolds through $p$. This has an advantage that the point $(0,0)$ is a critical point of the first return time function and then the cocycle can be defined as the value of Hessian on stable and unstable unit vectors. The disadvantage is that it is harder to see that the cocycle is H\"older. However, we will check that it makes little difference. Namely, the next lemma shows that the value of the mixed partial derivative at a periodic point does not depend on a particular choice of transversal. By the Livshits theorem~\cite{Liv} the cohomology class of a H\"older cocycle is determined by its values on periodic points. Hence, by the following lemma, cocycle $K$ is cohomologous to the one defined in~\cite{FH} and, hence, Foulon-Hasselblatt result indeed applies to the cocycle $K$ defined above.
\begin{lemma}
\label{lemma_periodic}
Let $p=X^T(p)$ be a periodic point and let $\cT_p$ be the transversal with adapted coordinates $(x,y)$. Let $\cT_p'$ be another transversal through $p$. Since $\cT_p$ and $\cT_p'$ are related by a short holonomy along the flow the adapted coordinates also induce coordinates on $\cT_p'$. Denote by $\xi$ the return time to $\cT_p$ and by $\xi'$ the return time to $\cT_p'$. Then
$$
\frac{\partial^2\xi}{\partial x\partial y}(0,0)=\frac{\partial^2\xi'}{\partial x\partial y}(0,0)
$$
\end{lemma}

\begin{proof}
For any point $q\in\cT_p$ we have $X^{u(q)}(q)\in\cT_p'$, where $u$ is a smooth function with $u(p)=0$. Then
$$
\xi'=\xi-u+u\circ F_p,
$$
where $F_p$ is the first return map to $\cT_p$. Recall that, when written in $(x,y)$-coordinates $F_p$ has a normal form with all second order terms vanishing~\eqref{eq_moser}. Hence taking the mixed partial derivative we have
\begin{multline*}
\frac{\partial^2\xi'}{\partial x\partial y}(0,0)=\frac{\partial^2\xi}{\partial x\partial y}(0,0)-\frac{\partial^2u}{\partial x\partial y}(0,0)+\frac{\partial}{\partial x}\left(\frac{\partial u}{\partial y}\circ F_p\,DF_p\left(\frac{\partial}{\partial y}\right)\right)(0,0)\\
= \frac{\partial^2\xi}{\partial x\partial y}(0,0)-\frac{\partial^2u}{\partial x\partial y}(0,0)+\frac{\partial^2u}{\partial x\partial y}(0,0)\mu\mu^{-1}+\frac{\partial u}{\partial y}(0,0)\,D^2F_p\left(\frac{\partial}{\partial x}, \frac{\partial}{\partial y}\right)(0,0)\\
=\frac{\partial^2\xi}{\partial x\partial y}(0,0)
\end{multline*}
\end{proof}

\subsection{Positive Proportion and Alternate Livshits Theorems}

We will need to use the following ``positive proportion version'' of the celebrated Livshits Theorem~\cite{Liv} which was recently established by Dilsavor and Marshall Reber~\cite{DMR}.

Let $X^t\colon M\to M$ be a transitive Anosov flow and let $\Delta$ be a fixed positive number. Let $\cP_T$ be the set of periodic orbits whose periods lie in the interval $(T, T+\Delta]$. Let $a\colon M\times\R\to\R$ be a H\"older continuous cocyle and $\cP_{T,a}\subset\cP_T$ be the subset of periodic orbits on which $a$ vanishes. The basic version of positive proportion Livshits theorem says that if $\limsup_{T\to\infty} \#\cP_{T,a}/\#\cP_T>0$ then $a$ is a coboundary. We will need a slightly more general version, where positive proportion is measured relative to an equilibrium state.

So let $B\colon M\times\R\to\R$ be another H\"older continuous cocycle and let $\mu_B$ be the associated equilibrium state~\cite{B}. Given a periodic orbit $\gamma$ we will write $B(\gamma)$ for the value $B(p,|\gamma|)$, where $p\in\gamma$ and $|\gamma|$ is the period of $p$. Also denote by $\delta_\gamma$ the invariant measure supported on $\gamma$ of total mass $|\gamma|$. Then the measures
$$
\mu_{T,B}=\frac{1}{\sum_{\gamma\in\cP_T}|\gamma|e^{B(\gamma)}} \sum_{\gamma\in\cP_T} e^{B(\gamma)}\delta_{\gamma}
$$
approximate $\mu_B$~\cite{bowen, F, P}. Formally, the set $\cP_T$ and, accordingly, the measures $\mu_{T,B}$ also depend on $\Delta$, but $\Delta$ will be fixed throughout the discussion, say one can take $\Delta=1$, and hence, it is safe to omit this dependence in notation to keep notation less cumbersome.

\begin{theorem}[Positive Proportion Livshits Theorem~\cite{DMR}]
\label{thm_ppl}
Let $X^t$, $B$ and $a$ be as described above. Assume that $\cP_{T,a}$ has positive proportion relative to $\mu_B$, that is,
$$
\limsup_{T\to\infty}\mu_{T,B}(\cP_{T,a})>0
$$
Then cocycle $a$ is a coboundary, that is,
$$
a(x,t)=u(X^t(x))-u(x)
$$
for some H\"older continuous function $u$.
\end{theorem}
\begin{remark} This theorem is formulated in a slightly different, but equivalent way in~\cite[Theorem~1.2]{DMR}. Namely, the approximating measures are defined using a different normalization
$$
\hat \mu_{T,B}=\frac{1}{\sum_{\gamma\in\cP_T}e^{B(\gamma)}} \sum_{\gamma\in\cP_T} \frac{e^{B(\gamma)}}{|\gamma|}\delta_{\gamma}
$$
and, accordingly, the positive proportion assumption in~\cite{DMR} is stated as $\limsup_{T\to\infty}\hat \mu_{T,B}(\cP_{T,a})>0$. We note that due to the obvious two-sided inequality
$$
\frac{T}{T+\Delta}\hat \mu_{T,B}(\cP_{T,a})\le \mu_{T,B}(\cP_{T,a}) \le \frac{T+\Delta}{T}\hat \mu_{T,B}(\cP_{T,a})
$$
these positive proportion assumptions are equivalent.
\end{remark}

In the course of the proof we will need to apply the above theorem two times. For the first application another version which we call Alternate Livshits Theorem suffices. This theorem allows for a different, quite elementary and soft proof based on Bowen's approximation formula for equilibrium states~\cite{bowen} which we give in the appendix.

\begin{theorem}[Alternate Livshits Theorem]
\label{thm_livshits}
Let $X^t\colon M\to M$ be a transitive Anosov flow and assume that $a_1, a_2, \ldots a_N\colon M\times \R\to \R$ are H\"older cocycles such that that for all periodic orbits $\gamma$ there exist an $i\in[1,N]$ such that $a_i$ vanishes on $\gamma$: $a_i(x,|\gamma|)=0$, $x\in \gamma$, $X^{|\gamma|}(x)=x$.
Then at least one of the cocycles is a coboundary, that is, there exists at least one $j\in[1,N]$ such that
$$
a_j(x,t)=u(X^t(x))-u(x)
$$
for some H\"older continuous function $u$.
\end{theorem}

Clearly, the Alternate Livshits Theorem is also a corollary of the Positive Proportion Livshits Theorem.
 
\section{Proof of Theorem~\ref{thm2}}

Here we will explain how Theorem~\ref{thm2} follows from Theorem~\ref{thm_tech}, which we proceed to state.

Recall that $X_1^t$ and $X_2^t$ are conjugate Anosov flows: $H\circ X_1^t=X_2^t\circ H$. Given a periodic point $p=X_1^T(p)$ of period $T$ let $\chi_1(p)$ be the positive Lyapunov exponent of $p$ and let $\chi_2((H(p))$ be the positive Lyapunov exponent of $H(p)$; that is
$$
\chi_1(p)=\frac1T\log J^uX_1^T(p),\,\,\,\, \chi_2(H(p))=\frac1T\log J^uX_2^T(H(p)),
$$
The following is a local result which is crucial for the proof of Theorem~\ref{thm2}.
\begin{theorem}
\label{thm_tech}
Let $X_1^t$, $X_2^t$ and $H$ be as in Theorem~\ref{thm2}. Then for every periodic point $p=X_1^T(p)$ of period $T$ the following tetrachotomy holds
\begin{itemize}
\item either $K_1(p, T)=K_2(H(p),T)=0$, where $K_i$ is the longitudinal Anosov cocycle of $X_i^t$, $i=1,2$;
\item or $K_1(p,T)=0$, $K_2(H(p),T)\neq0$ and $\chi_1(p)<\chi_2(H(p))$;
\item or $K_2(H(p),T)=0$, $K_1(p,T)\neq0$ and $\chi_1(p)>\chi_2(H(p))$;
\item or $K_1(p,T)\neq0$, $K_2(H(p),T)\neq0$ and $\chi_1(p)=\chi_2(H(p))$
\end{itemize}
\end{theorem}

Now consider the cocycle $A(x,t)=\log J^uX_1^t(x)-\log J^uX_2^t(H(x))$ over $X_1^t$. Then, according to the above theorem, we have that over every periodic orbit at least one of the cocycles $K_1$, $K_2\circ H$ or $A$ vanishes. 
 Then the Alternate Livshits Theorem applies to give that at least one of these cocycles is a coboundary.

If $A$ is a coboundary, then all Lyapunov exponents at periodic points match under conjugacy and then de la Llave-Marco-Moriy\'on-Pollicott theorem applies and yields $C^{r_*}$ regularity of the conjugacy.

Hence we only need to consider the case when $K_1$ is a coboundary over $X_1^t$. (If $K_2\circ H$ is coboundary over $X_1^t$ then $K_2$ is coboundary over $X_2^t$ which is an entirely symmetric situation.) By the Foulon-Hasselblatt theorem~\cite{FH}, we conclude that $X_1^t$ is either a contact flow or a constant roof suspension. If $X_1^t$ is a constant roof suspension flow then, in fact, the second flow also has to be a constant roof suspension. Indeed, in this case $M$ is the mapping torus of a hyperbolic automorphism and the second flow also has a global torus section, since the flows are conjugate. Then, since the periods match, the roof function of the second flow has exactly the same sums over periodic orbits as the constant roof of the first flow. Hence, by Livshits theorem, this roof function is cohomologous to the same constant which means that the second flow is also a constant roof suspension. This gives us one of the alternative conclusions of Theorem~\ref{thm2}.

Thus it remains to consider the case when $X_1^t$ is a contact flow. The following proposition completes the proof of Theorem~\ref{thm2} modulo Theorem~\ref{thm_tech}. We note that this proposition improves a theorem of Feldman and Ornstein who proved that a pair of 3-dimensional contact Anosov flows are $C^0$ conjugate if and only if they are $C^1$ conjugate~\cite{FO}.

\begin{proposition}
 Let $X_i^t\colon M\to M$, $i=1,2$, be $C^r$, $r>2$, Anosov flows which are conjugate via a  conjugacy $H$. Assume additionally that $X_1^t$ is a contact flow and that $X_2^t$ is volume preserving. Then $X_2^t$ is also contact and the conjugacy is $C^{r_*}$ regular. 
 \end{proposition}

\begin{proof}
Our objective is to show that the longitudinal Anosov cocycle $K_2$ of $X_2^t$ is a coboundary. Then, by the Foulon-Hasselblatt theorem we have that $X_2^t$ is either contact or a constant roof suspension. Since $X_1^t$ is contact the case when $X_2^t$ is a constant roof suspension is easily ruled out as explained in the paragraphs preceding the proposition.
Hence, we have that both $X_1^t$ and $X_2^t$ are contact and we can use Feldman-Ornstein theorem~\cite{FO} to conclude that $H$ is a $C^1$ diffeomorphism. Then one can use de la Llave bootstrap~\cite{dlL} to gain optimal smoothness $C^{r_*}$. (Also see~\cite{GRH0} for a refined Feldman-Ornstein argument which gives optimal smoothness right away for conjugacy of contact flows.)

Denote by $\omega$ the invariant volume for $X_2^t$ and let $B(x,t)=-\log J^uX_2^t$. Recall that $\omega$ is the equilibrium state for $B$~\cite{B}.

As before, let $\cP_T$ be the set of periodic orbits of $X_2^t$ whose periods lie in the interval $(T, T+\Delta]$ and let
$$
\cP_{T,K_2}=\{\gamma\in\cP_T: \, K_2(p,|\gamma|)=0, p\in\gamma\}
$$
If $\limsup_{T\to\infty}\mu_{T,B}(\cP_{T,K_2})>0$ then $K_2$ is a coboundary by the Positive Proportion Livshits Theorem (Theorem~\ref{thm_ppl}). Hence we need to rule out the following possibility
\begin{equation}
\label{eq_zero}
\lim_{T\to\infty}\mu_{T,B}(\cP_{T,K_2})=0
\end{equation}

Since $K_1$ vanishes on every periodic orbit $\gamma$ of $X_1^t$, by Theorem~\ref{thm_tech}, we have that if $K_2$ does vanish on $H(\gamma)$ then $\chi_1(p)<\chi_2(H(p))$, which justifies the following notation
$$
\cP_{T,\chi_1<\chi_2}=\cP_T\backslash \cP_{T,K_2}
$$
We can decompose the approximating measures $\mu_{T, B}$ accordingly
\begin{multline*}
\mu_{T,B}=\frac{1}{\sum_{\gamma\in\cP_T}|\gamma|e^{B(\gamma)}} \left(\sum_{\gamma\in\cP_{T, K_2}} e^{B(\gamma)}\delta_{\gamma}+\sum_{\gamma\in\cP_{T, \chi_1<\chi_2}} e^{B(\gamma)}\delta_{\gamma}\right)\\
=\mu_{T,B}(\cP_{T,K_2})\frac{1}{\mu_{T,B}(\cP_{T,K_2})}\sum_{\gamma\in\cP_{T, K_2}} e^{B(\gamma)}\delta_{\gamma}\\+
\mu_{T,B}(\cP_{T,\chi_1<\chi_2})\frac{1}{\mu_{T,B}(\cP_{T,\chi_1<\chi_2})}\sum_{\gamma\in\cP_{T, \chi_1<\chi_2}} e^{B(\gamma)}\delta_{\gamma}\\
=\mu_{T,B}(\cP_{T,K_2})\mu_{T,B,K_2}+\mu_{T,B}(\cP_{T,\chi_1<\chi_2})\mu_{T,B,\chi_1<\chi_2},
\end{multline*}
where $\mu_{T,B,K_2}$ and $\mu_{T,B,\chi_1<\chi_2}$ are defined by the last line. In this way, we have a decomposition of $\mu_{T,B}$ as a convex combination of probability measures $\mu_{T,B,K_2}$ and $\mu_{T,B,\chi_1<\chi_2}$. By~\eqref{eq_zero} the coefficients of this decomposition converge to 0 and 1, respectively. Hence, taking the limit as $T\to\infty$ yields
\begin{equation}
\label{eq_wc}
\lim_{T\to\infty}\mu_{T,B,\chi_1<\chi_2}=\lim_{T\to\infty}\mu_{T,B}=\mu_B=\omega
\end{equation}

Now we are ready to make the estimates. We have
\begin{multline*}
\int_M\log J^uX^1_1\circ H^{-1}d\mu_{T,B,\chi_1<\chi_2}\\
=\frac{1}{\mu_{T,B}(\cP_{T,\chi_1<\chi_2})}\sum_{\gamma\in \cP_{T,\chi_1<\chi_2}}e^{B(\gamma)}\int_{H^{-1}(\gamma)}\log J^uX_1^1d\delta_{H^{-1}(\gamma)}
\\=\frac{1}{\mu_{T,B}(\cP_{T,\chi_1<\chi_2})}\sum_{\gamma\in \cP_{T,\chi_1<\chi_2}}e^{B(\gamma)}\chi_1(H^{-1}(\gamma))\\
<\frac{1}{\mu_{T,B}(\cP_{T,\chi_1<\chi_2})}\sum_{\gamma\in \cP_{T,\chi_1<\chi_2}}e^{B(\gamma)}\chi_2(\gamma)\\
=\frac{1}{\mu_{T,B}(\cP_{T,\chi_1<\chi_2})}\sum_{\gamma\in \cP_{T,\chi_1<\chi_2}}e^{B(\gamma)}\int_{\gamma}\log J^uX_2^1d\delta_{\gamma}\\
=\int_M\log J^uX_2^1d\mu_{T,B,\chi_1<\chi_2}
\end{multline*}
Taking the limit of both sides of this inequality and using weak$\,^*$ convergence~\eqref{eq_wc} we obtain
$$
\int_M\log J^uX^1_1\circ H^{-1}d\omega\le \int_M\log J^uX_2^1d\omega
$$

On the other hand, using the Jacobian formula for the positive Lyapunov exponent, the Pesin formula and the Margulis-Ruelle inequality we have
\begin{multline*}
\int_M\log J^uX_2^1d\omega=\chi_2(\omega)=h(\omega, X_2^1)
=h(H^*\omega, X_1^1)\\
\le \chi_1(H^*\omega)
=\int_M\log J^uX_1^1dH^*\omega=\int_M\log J^uX^1_1\circ H^{-1}d\omega
\end{multline*}
We have arrived at opposing inequalities, hence, both must be equalities. In particular, equality is achieved in the Margulis-Ruelle inequality $h(H^*\omega, X_1^1)
\le \chi_1(H^*\omega)$ which can happen if and only if $H^*\omega$ is an absolutely continuous invariant measure. Recall from the discussion in Section~2.1 that if $H$ sends an absolutely continuous invariant measure to an absolutely continuous invariant measure then $H$ is smooth. Hence we have that $X_2^t$ is also contact and has vanishing longitudinal Anosov cocyle which rules out~\eqref{eq_zero}.
\end{proof}

\section{Proof of Theorem~\ref{thm_tech}}

\label{sec4}

In this section we prove Theorem~\ref{thm_tech}, which then completes the proof of Theorem~\ref{thm2}.

It will become clear from the proof that while this is a theorem about two flows, it is, in fact, secretly, a theorem about a single flow $X^t$. Namely, if the longitudinal Anosov cocycle is non-zero at a periodic orbit $\gamma$ then we recover the Lyapunov exponent of $\gamma$ from a sequence of periods of periodic orbits which approximate a homoclinic orbit of $\gamma$. Then, since for conjugate flows periods of corresponding periodic orbits are equal, this allows us to conclude matching of Lyapunov exponents.\footnote{We would like to thank Martin Leguil who pointed out to us that such connections between periods and Lyapunov exponents exist. In particular, similar formulas were extensively used in the billiards setting~\cite{HKS, BDKL,DKL}. This approach allowed us to replace our earlier ``fractal graph argument'' with an easier and shorter approximation argument.}

Let $X^t$ be a 3-dimensional volume preserving Anosov flow and let $\gamma$ be a periodic orbit with a base-point $p\in\gamma$ and period $T_0$. Let $\cT$ be a $C^r$ smooth transversal through $p$ which contains stable and unstable manifolds of $p$. We denote by $F\colon F^{-1}(\cT)\to \cT$ the local return map to $\cT$ and by $\tau+T_0\colon F^{-1}(\cT)\to \R$ the return time to $\cT$, that is, $\tau$ is defined by $X^{\tau(x)+T_0}(x)=F(x)$. Because stable and unstable manifolds are invariant we have that $\tau$ vanishes on $W^s_{loc}(p)\cup W^u_{loc}(p)$, where $W^s_{loc}(p)$ and $W^u_{loc}(p)$ are connected components of $p$ of $W^s(p)\cap\cT$ and $W^u(p)\cap\cT$, respectively. 
We equip $\cT$ with Moser coordinates so that $F$ has the form~\eqref{eq_moser}
\begin{equation*}
F(x,y)=(\mu x+xy\phi_1(x,y), \mu^{-1}y+xy\phi_2(x,y) ).
\end{equation*}
Recall that in these coordinates we still have that both $F$ and $\tau$ are $C^{r}$ smooth. As before, we let $\theta=\min\{1,r-2\}$. Since we assumed that $r>2$ we have $\theta>0$. Recall that $\phi_i(x,y)\le C(|x|^\theta+|y|^\theta)$.

Now consider any orbit $\cO$ homoclinic (that is, bi-asymptotic) to the orbit of $p$. Let $p_{in}$ be the first point (with respect to time order on $\cO$) in the intersection $\cO\cap W^s_{loc}(p)$ and let $p_{out}$ be the last point in $\cO\cap W^u_{loc}(p)$. Then we have $X^{T'}(p_{out})=p_{in}$ and the orbit segment $\{X^t(p_{out}): t<0<T'\}$ is disjoint with $\cT$.

\begin{figure}[ht]
\centering
\includegraphics[width=\textwidth]{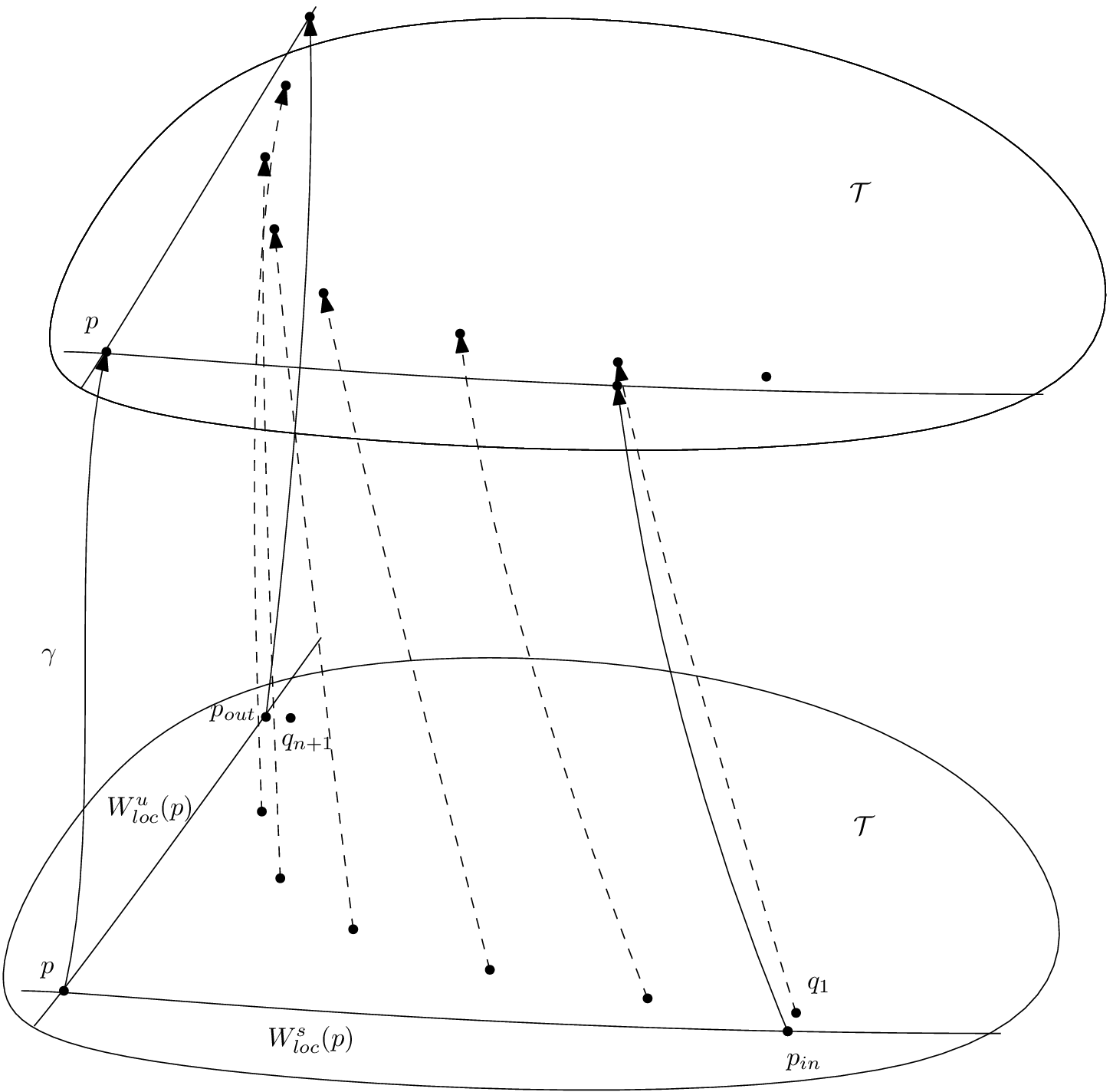}
\end{figure}

The forward orbit $\{F^i(p_{in}); i\ge 0\}$ converges to $p$ along $W^s_{loc}(p)$ and the backward orbit $\{F^{-i}(p_{out}); i\ge 0\}$ converges to $p$ along $W^u_{loc}(p)$. Hence the segment of $\cO$ from $F^{-i_1}(p_{out})$ to $F^{i_2}(p_{in})$ forms a pseudo-orbit for $X^t$ which can be shadowed by a periodic orbit $\gamma_n$ according to the Anosov closing lemma. By appropriately choosing $i_1$ and $i_2$ with $|i_1-i_2|\le 1$ we can arrange that $\gamma_n$ intersects $\cT$ at $n+1$ points $q_1, q_2, \ldots q_{n+1}$ (ordered with respect to time direction on $\gamma_n$) with $q_1$ being close to $p_{in}$ and $q_{n+1}$ being close to $p_{out}$. We denote by $(x_i,y_i)$ the coordinates of $q_i$ and by $x_{in}$ and $y_{out}$ the $x$-coordinate of $p_{in}$ and the $y$-coordinate of $p_{out}$, respectively. From the shadowing property it is clear that $q_1$ is very close to $p_{in}$ and $q_{n+1}$ is very close to $p_{out}$. The next lemma makes it quantitative.
\begin{lemma} 
\label{lemma1}
The coordinates of $q_1$ and $q_{n+1}$ satisfy the following estimates with constants uniform in $n$
$$
c_1\mu^n\le |y_1|\le c_2\mu^n,\,\,\, c_1\mu^n\le |x_1-x_{in}|\le c_2\mu^n,\,\,\, 
$$
and
$$
c_1\mu^n\le |x_{n+1}|\le c_2\mu^n,\,\,\, c_1\mu^n\le |y_{n+1}-y_{out}|\le c_2\mu^n.\,\,\, 
$$
\end{lemma}

\begin{proof} Probably the simplest way to verify these inequalities is to use $C^1$-linearization. It is well-known that $\cT$ admits $C^1$ coordinate system $(\tilde x,\tilde y)$ which makes dynamics fully linear~\cite{H, Bel}
$$
F(\tilde x,\tilde y)=(\mu\tilde x,\mu^{-1}\tilde y).
$$
This coordinate change has the form
$$
(\tilde x,\tilde y)=(x+x\psi_1(x,y), y+y\psi_2(x,y)).
$$
where $\psi_1$ and $\psi_2$ are continuous (and hence bounded) functions on $\cT$.

Since the points $q_{n+1}$ converge to $p_{out}$ as $n\to\infty$ we have $\bar c_1\le |\tilde y_{n+1}|\le\bar c_2$ for some positive $c_1$ and $c_2$ and all $n$. Then iterating $n$ times backwards with linear dynamics yields $\bar c_1\mu^n\le |\tilde y_{1}|\le\bar c_2\mu^n$. Observing that $|\tilde y_1/y_1|$ is uniformly bounded above and  below yields the first posited inequality. The second inequality follows from similar considerations and the remaining inequalities also follow using the same argument by iterating forward instead.
\end{proof}

From the lemma we have $\|q_1-p_{in}\|\le C\mu^n$ and $\|q_{n+1}-p_{out}\|\le C\mu^n$. Denote by $T_n$ the period of $\gamma_n$ and by $T_n'$ the time from $q_{n+1}$ to $q_1$, that is, $X^{T'_n}(q_{n+1})=q_{1}$. Then, since $\cT$ is transverse to the flow, we have
\begin{equation}
\label{eq_t_prime}
|T_n'-T'|\le C (\|q_{n+1}-p_{out}\|+\|q_1-p_{in}\|)\le C\mu^n
\end{equation}

\begin{lemma} 
\label{lemma_2}
We have the following asymptotic formula for $T'$
$$
T'=\lim_{n\to\infty} T_n-nT_0
$$
Hence the value $T'$ is determined by the periods of periodic orbits.
\end{lemma}

\begin{proof}
Because $\tau$ vanishes on $W^s_{loc}(p)\cup W^u_{loc}(p)$ we have that $|\tau(x,y)|\le c|xy|$. Since orbit $\gamma_n$ intersects $\cT$ $n+1$ times we can write $T_n$ as the following sum of $n+1$ terms.
$$
T_n=T_n'+\sum_{i=1}^n(\tau(x_i,y_i)+T_0)\le T_n'+nT_0+C\sum_{i=1}^n c|x_iy_i|.
$$
Here $(x_i,y_i)$ are the coordinates of $q_i$. Since by~\eqref{eq_t_prime} we have $T_n'\to T'$ as $n\to\infty$, to establish the lemma it remains to prove that the above sum converges to 0. In fact, we will prove that $\sum_{i=1}^n|x_iy_i|=O(n\mu^n)$ as $n\to\infty$ which will also be helpful for the next lemma.

First note that we have $|x_i|\le C(\mu^+)^i$ and $|y_i|\le C(\mu^+)^{n-i}$ for some $C>0$ and some $\mu^+\in(\mu,1)$. Indeed, since $F$ is $C^1$ close to the linear map $(x,y)\mapsto (\mu x,\mu^{-1}y)$, we have that $x_{i+1}\le\mu^+x_i$ and $y_i\le \mu^+y_{i+1}$ for all $(x_i,y_i)$ which are sufficiently close to the origin. Hence, by taking a smaller transversal $\cT$ (or, equivalently, replacing $p_{in}$ and $p_{out}$ with $F^k(p_{in})$ and $F^{-k}(p_{out})$, respectively, for some large $k$) we can assume that we have such exponential bounds on $x_i$ and $y_i$.

Recall that we have $|x_1y_1|\le C\mu^n$ by Lemma~\ref{lemma1} and we would like to bound the products $|x_iy_i|$ for all $i=1,\ldots n$. We can do so using induction
\begin{multline*}
|x_{i+1}y_{i+1}|=|(\mu x_i+x_iy_i\phi_1(x_i,y_i))(\mu^{-1}y_i+x_iy_i\phi_2(x_i,y_i))|\\
=|x_{i}y_{i}||1+\mu x_i\phi_2(x_i,y_i)+\mu^{-1}y_i\phi_1(x_i,y_i)+x_iy_i\phi_1(x_i,y_i)\phi_2(x_i,y_i)|\\
\le |x_{i}y_{i}|(1+C(\mu^+)^i+C(\mu^+)^{n-i}).
\end{multline*}
From convergence of geometric series, it is a standard calculus exercise to check that the products
$$
\prod_{i=1}^n(1+C(\mu^+)^i+C(\mu^+)^{n-i})
$$
are bounded uniformly in $n$. Therefore, by induction, we have
\begin{equation*}
|x_iy_i|\le C_2\mu^n
\end{equation*}
for some $C_2>0$ and for all $i=1,\ldots n$ and $n\ge 1$. Hence, $\sum_{i=1}^n|x_iy_i|=O(n\mu^n)$. Also note that by Lemma~\ref{lemma1} $|x_1y_1|\ge c\mu^n$ and we also have
$$
|x_{i+1}y_{i+1}|\ge  |x_{i}y_{i}|(1-C(\mu^+)^i-C(\mu^+)^{n-i}),
$$
which, again by induction, implies  a uniform lower bound. Hence we, in fact, have
\begin{equation}
\label{eq_xy_bound}
C_1\mu^n\le|x_iy_i|\le C_2\mu^n
\end{equation}
with some $C_1$ and $C_2$ which do not depend on $n$.
\end{proof}

\begin{remark} Alternatively, this last argument could  utilize $C^1$ linearization again. Namely, for fully linearized system $|\tilde x_i\tilde y_i|$ is independent of $i$ and proportional to $\mu^n$ and we have  $|x_iy_i|\asymp|\tilde x_i\tilde y_i|$.
\end{remark}

\begin{lemma} 
\label{lemma_4}
The periods $T_n$ of the periodic orbits $\gamma_n$ admit the following asymptotic expansion
$$
T_n=nT_0+T'+c_nKn\mu^n+O(\mu^n),
$$
where $K$ is value of the longitudinal Anosov cocycle on $\gamma$ and the sequence of constants $\{c_n; n\ge 1\}$ is uniformly bounded above and below.
\end{lemma}
\begin{remark} Using more delicate estimates one can actually obtain a true asymptotic formula $T_n=nT_0+T'+c_0Kn\mu^n+O(\mu^n)$ where $c_0\neq 0$, however the above lemma is easier to establish and it is sufficient for our purposes.
\end{remark}
\begin{proof}
We have $T_n=T_n'+nT_0+\sum_{i=1}^n\tau(x_i,y_i)$. Using~\eqref{eq_t_prime} we can write
$$
T_n=nT_0+T'+\sum_{i=1}^n\tau(x_i,y_i)+O(\mu^n).
$$
Recall that $\tau$ vanishes to the first order at $p$ and we can expand it as follows
$$
\tau(x,y)=Kxy+xyB(x)+xy C(y)+O((xy)^{1+\theta}),
$$
where $B(x)=O(x^\theta)$ and $C(y)=O(y^\theta)$. Also recall also that $K$ is the value of longitudinal Anosov cocycle $K=K(p,T_0)=\frac{\partial^2\tau}{\partial x\partial y}(0,0)$. 

Now we will split the sum $\sum_{i=1}^n\tau(x_i,y_i)$ into four sums according to the above expansion for $\tau$. In order not to write absolute value bars we can assume that the orbit $\{q_i, i=1,.. n+1\}$ belongs to the first quadrant so that all $x_i$ and $y_i$ are positive (if the orbit belongs to a different quadrant then we can change the orientation of axes accordingly). First, using~\eqref{eq_xy_bound} we have
$$
C_1Kn\mu^n\le\sum_{i=1}^nKx_iy_i\le C_2 Kn\mu^n
$$
Hence, we indeed have that $\sum_{i=1}^nKx_iy_i=c_nKn\mu^n$, where $c_n\in[C_1,C_2]$, $C_1>0$.

Note that to prove the posited formula it remains to show that the remaining three sums are $O(\mu^n)$. Clearly, $\sum_{i=1}^n O((x_iy_i)^{1+\theta})=nO(\mu^{n(1+\theta)})=o(\mu^n)$. Next we have 
$$
\left|\sum_{i=1}^n x_iy_iB(x_i)\right|\le C_2\mu^n \sum_{i=1}^n C x_i^{\theta}\le C_3\mu^n\sum_{i=1}^n(\mu^+)^{i\theta}
$$
Since the latter sum are summable geometric series, we  obtain $\sum_{i=1}^n x_iy_iB(x_i)=O(\mu^n)$. Analogously, we have $\sum_{i=1}^n x_iy_iC(y_i)=O(\mu^n)$, which finishes the proof.
\end{proof}

We can now go back to our setting of conjugate Anosov flows $X_1^t$ and $X_2^t$. We will apply Lemma~\ref{lemma_4} to a periodic orbit $\gamma_1$ of $X_1^t$ and $\gamma_2=H(\gamma_1)$. Note that if $\{\gamma_n\}$ is a sequence of periodic orbits approximating a homoclinic orbit of $\gamma_1$ then  the sequence $\{H(\gamma_n)\}$ approximates a homoclinic orbit of $\gamma_2$. And Lemma~\ref{lemma_4} yields the formula
$$
T_n-nT_0-T'=c_n^jnK_j\mu_j^n+O(\mu_j^n), \,\,\, j=1,2,
$$
where $K_j$ is the value of the longitudinal Anosov cocycle on $\gamma_j$, $\mu_j$ is the smaller eigenvalue of the Poincar\'e map at $\gamma_j$ and $c_n^j$ are some constants uniformly bounded from above and away from zero. Note that $T_0$ and $T_n$ do not have a $j$ subscript because these are lengths of periodic orbits which are the same for both flows since they are conjugate. Also, by Lemma~\ref{lemma_2}, the value of $T'$ is determined by the lengths of periodic orbits, hence, is the same for both flows.

First consider the case when both $K_1$ and $K_2$ do not vanish. In this case we can recover the multipliers $\mu_j$, $j=1,2$, from the periods. Indeed, taking logarithms we have
$$
\log(T_n-nT_0-T')=n\log\mu_j+\log(c_n^jK_jn+O(1)),\,\, j=1,2
$$
Note that, since $K_j\neq0$, we have $c_n^jK_jn+O(1)>0$ for all sufficiently large $n$ and it grows sublinearly. Hence, dividing by $n
$ and taking the limit gives
$$
\log\mu_1=\log\mu_2=\lim_{n\to\infty}\log(T_n-nT_0-T')
$$
Hence the unstable exponents $\chi_1(\gamma_1)=-\log\mu_1$ and $\chi_2(\gamma_2)=-\log\mu_2$ are equal, yielding the last alternative conclusion of Theorem~\ref{thm_tech} in this case.

The case $K_1=K_2=0$ gives the first alternative conclusion of the theorem. Hence we are left consider the second case when $K_1=0$ and $K_2\neq 0$, the remaining third case ($K_1\neq 0$ and $K_2=0$) being fully analogous. In this case we obtain
$$
O(\mu_1^n)=T_n-nT_0-T'=c_n^2nK_2\mu_2^n+O(\mu_2^n)
$$
Hence $n\mu_2^n=O(\mu_1^n)$, which implies that $\mu_1>\mu_2$ and $\chi_1(\gamma_1)=-\log\mu_1<-\log\mu_2=\chi_2(\gamma_2)$.

\section{Proof of Theorem~\ref{thm3}}

We begin with the definition the generalized longitudinal Anosov cocycle. Let $X^t\colon M\to M$ be a volume preserving 3-dimensional Anosov flow and let $\phi\colon M\to\R$ be a function.
Let $\cT_p$, $p\in M$, be the system of adapted transverals for $X^t$ equipped with $(x,y)$-coordinates as defined in Section~\ref{sec23}. For any $p\in M$ and $t\in\R$ let $\xi\colon \cT_p\to\R$ be the first return time from $\cT_p$ to $\cT_{X^t(p)}$. Consider the function $\xi_\phi\colon \cT_p\to\R$ given by
$$
\xi_\phi(q)=\int_0^{\xi(q)}\phi(X^t(q))dt
$$
Define
$$
K_\phi(p,t)=\frac{\partial^2\xi_\phi}{\partial x\partial y}(0,0)
$$
When we need to emphasize dependence on the flow we will additonally use a superscript $K_\phi^X$. The following properties of the generalized longitudinal Anosov cocyle are immediate from the definition.
\begin{enumerate}
\item If $\phi\equiv 1$ then $\xi_\phi=\xi$ and $K_\phi=K$ is the usual longitudinal cocycle defined earlier~\eqref{eq_AC}.
\item Given two functions $\phi,\psi\colon M\to \R$ and two constants $b,c\in\R$ we have $K_{b\phi+c\psi}=bK_{\phi}+cK_{\psi}$.
\item If $\phi>0$ and $Y^t$ is a reparametrization of $X^t$ with generator $Y=\frac1\phi X$ then $K^Y(p, t_\phi)=K^Y_1(p, t_\phi)=K^X_\phi(p,t)$, where $t_\phi=\int_0^t\phi (X^s(p))ds$ (cf. Remark~2.1). In particular, if $K_\phi^X$ is a coboundary over $X^t$ then $K^Y$ is a coboundary over $Y^t$. Indeed, 
if $K_\phi^X(p,t)=u(p)-u(X^t(p))$ then $K^Y(p, t_\phi)=u(p)-u(X^t(p))=u(p)-u(Y^{t_\phi}(p))$.
\end{enumerate}

We have the following generalization of Theorem~\ref{thm_tech}.
\begin{theorem}
\label{thm_tech2}
Let $X_1^t$, $X_2^t$, $\phi_1$, $\phi_2$ and $H$ be as Theorem~\ref{thm3}. Then for every periodic point $p=X_1^T(p)$ of period $T_p$ the following pentachotomy holds
\begin{itemize}
\item either 
$$\int_0^{T_p}\phi_1(X_1^t(p))dt=0;$$
\item or $K_{\phi_1}(p, T_p)=K_{\phi_2}(H(p),T_{H(p)})=0$, where $K_{\phi_i}$ is the generalized longitudinal Anosov cocycle of $(X_i^t,\phi_i)$, $i=1,2$;
\item or $K_{\phi_1}(p,T_p)=0$, $K_{\phi_2}(H(p),T_{H(p)})\neq0$ and $\chi_1(p)<\chi_2(H(p))$;
\item or $K_{\phi_2}(H(p),T_{H(p)})=0$, $K_{\phi_1}(p,T_p)\neq0$ and $\chi_1(p)>\chi_2(H(p))$;
\item or $K_{\phi_1}(p,T_p)\neq0$, $K_{\phi_2}(H(p),T_{H(p)})\neq0$ and $\chi_1(p)=\chi_2(H(p))$.
\end{itemize}
\end{theorem}

One way to establish Theorem~\ref{thm_tech2} is to carefully repeat all the arguments of Section~\ref{sec4} while replacing the time with appropriate integrals of $\phi_i$. While majority of the arguments remain the same, some steps would require substantial modification. An alternative way is to reduce Theorem~\ref{thm_tech2} to Theorem~\ref{thm_tech}, which is what we will do below.

\begin{proof}
We will denote by $\gamma$ the periodic orbit through the point $p$. We will denote by $T_0=T_p$ the period of $p$ to be consistent with notation used in Section~\ref{sec4}. Also for the rest of the proof we will write $X^t$ and $\phi$ instead of $X_1^t$ and $\phi_1^t$ as the bulk of the proof only considers the first flow.

We can assume that $\int_0^{T_0}\phi(X^t(p))dt>0$. Indeed, if the integral vanishes then it puts us in the first alternative of the theorem and if it is negative we can replace the matching pair $(\phi_1,\phi_2)$ with the matching pair $(-\phi_1,-\phi_2)$ and then argue in exactly same way since the integral of $-\phi_1$ is positive.

As in Section~\ref{sec4} we consider the heteroclinic orbit $\cO$ and a sequence of periodic orbits $\gamma_n$, $n\ge 1$ which approximate $\cO$.

\begin{lemma}
\label{lemma_positive}
The function $\phi$ is cohomologous to a function $\bar\phi$ which is positive on $\gamma\cup\cO$.
\end{lemma}

\begin{proof} Since the integral of $\phi$ over $\gamma$ is positive the function
$$
\tilde \phi(x)=\frac1T\int_0^T\phi(X^t(x))dt
$$
is positive on $\gamma$ if $T$ is chosen to be sufficiently large. The function $\tilde\phi$ is cohomologous to $\phi$ by a standard calculation. Since $\cO$ is bi-asymptotic to $\gamma$ we have that $\tilde \phi$ is positive on all but a finite piece of $\cO$ of length $L$. To make it positive there as well we can repeat the trick and set

$$
\bar \phi(x)=\frac1T\int_0^T\tilde\phi(X^t(x))dt
$$
If $T\ggg L$ then $\bar\phi$ will be positive on all of $\cO$ and, clearly, stays positive on $\gamma$.
\end{proof}

Let $\cU$ be a neighborhood of $\gamma\cup\cO$ such that $\bar\phi|_\cU>0$. Pick any positive smooth function $\hat\phi\colon M\to\R$ such that $\hat\phi|_\cU= \bar\phi|_\cU$. Define
$$
T_0^\phi= \int_{\gamma}\phi dt,\,\,\,\,\,\, T_n^{\phi}=\int_{\gamma_n}\phi dt,\,\, n\ge 1,
$$
and similarly $T_n^{\bar\phi}$ and $T_n^{\hat\phi}$. Since $\phi$ is cohomologous to $\bar\phi$ we have $T_n^\phi=T_n^{\bar\phi}$ and, since $\gamma_n\subset \cU$ for all sufficiently large $n$, we also have that $T_n^{\bar\phi}= T_n^{\hat\phi}$ for all sufficiently large $n$.

Now consider reparametrization $Y^t$ with generator $Y=\frac{1}{\hat\phi}X$. Note that this reparametrization is well defined because $\hat\phi>0$. Also note that the $Y^t$-period of $\gamma_n$ is given by $T_n^{\hat \phi}$. We can apply Lemma~\ref{lemma_4} to $Y^t$ and the sequence of periodic orbits $\gamma_n$ considered as periodic orbits of $Y^t$ to obtain the following.

\begin{lemma} 
The periods $T_n^{\hat\phi}$ of the periodic orbits $\gamma_n$ admit the following asymptotic expansion
$$
T_n^{\hat\phi}=nT_0^{\hat\phi}+T'+c_nK^Yn\mu^n+O(\mu^n),
$$
where $T'$ is a certain number determined by $\{T_n^{\hat\phi}, n\ge 0\}$, $K^Y=K^Y(p, T_0^{\phi})$ is value of the longitudinal Anosov cocycle on $\gamma$, $\mu\in(0,1)$ is the eigenvalue of the return map at $p$, and the sequence of constants $\{c_n; n\ge 1\}$ is uniformly bounded above and below.
\end{lemma}
\begin{remark} We recall that the fact that $T'$ can be recovered from the periods is the contents of Lemma~4.2.
\end{remark}

Recalling the behaviour of generalized longitudinal Anosov cocycle under reparametrizations and combining with above observations, the asymptotic formula of Lemma~5.3 can be rewritten in the following way for all sufficiently large $n$:
$$
T_n^{\phi}=nT_0^{\phi}+T'+c_nK^X_\phi(p, T_0)n\mu^n+O(\mu^n)
$$
This formula allows to recover the eigenvalue at $p$ from the sequence $\{T_n^{\phi}, n\ge 0\}$ if $K^X_\phi(p, T_0)\neq 0$. Using this observation, the same formula for the second flow $X_2^t$ and the matching $T_n^{\phi_1}=T_n^{\phi_2}$, one obtains the alternate conclusions of Theorem~\ref{thm_tech2} in exactly the same way as in the end of the proof of Theorem~\ref{thm_tech} (the arguments at the end of Section~\ref{sec4} after the proof of Lemma~\ref{lemma_4}). This finishes the proof of Theorem~\ref{thm_tech2}.
\end{proof}

We will need one more simple Lemma.

\begin{lemma}
\label{lemma_contact}
Let $Y^t\colon M\to M$ be a reparametrization of a contact flow whose longitudinal Anosov cocycle is trivial. Then $Y^t$ is also contact.
\end{lemma}

\begin{proof}
The Foulon-Hasselblatt theorem says that triviality of the longitudinal Anosov cocycle implies that $Y^t$ is either contact or a constant roof suspension of an Anosov diffeomorphism of $\T^2$. To rule out the latter case recall that contact flows are homologically full~\cite[Thorem~2.9]{GRH2}. The property of being homologically full persists under reparametrizations and suspension flows are not homologically full. Hence $Y^t$ cannot be a constant roof suspension flow.
\end{proof}

Now we proceed with the proof of Theorem~\ref{thm3}. 

By the Alternate Livshits Theorem, from Theorem~\ref{thm_tech2} we have that either $\phi_1$ is a coboundary or at least one of the cocycles $K_{\phi_1}$ and $K_{\phi_2}$ is a coboundary. If $\phi_1$ is coboundary then from the matching assumption $\phi_2$ integrates to zero over periodic orbits of $X_2^t$ and, hence, by the Livshits Theorem, we also have that $\phi_2$ is a coboundary, which finishes the proof in this case.
 
Hence we can now assume that $K_{\phi_1}$ is a coboundary (the case when $K_{\phi_2}$ is a coboundary is entirely symmetric). Pick a constant $c_0$ such that $\phi_1+c_0>0$. Since $X_1^t$ is contact we have that $c_0K=c_0K_1=K_{c_0}$ is a coboundary.
Hence, $K_{\phi_1+c_0}=K_{\phi_1}+K_{c_0}$ is also a coboundary.

Since $\phi_1+c_0>0$ we can consider the reparametrization $Y_1^t$ with generator $Y_1=\frac{1}{\phi_1+c_0}X_1$. By property~3 above we have $K^Y=K^X_{\phi_1+c_0}$. Hence the longitudinal Anosov cocyle of $Y^t$ is trivial and, using Lemma~\ref{lemma_contact}, we conclude that $Y^t$ is also a contact flow.

Now we apply~\cite[Theorem~7.1]{GRH2} which says that since $Y_1^t$ is a contact reparametrization of a contact flow $X_1^t$ then $\phi_1+c_1=C+\omega(X_1)$, where $\omega$ is a closed 1-form. Hence we have $\phi_1=c_0+\omega(X_1)$ for some constant $c_0$.

If $c_0=0$ then $\phi_1$ is an abelian a coboundary. Further, for any homologically trivial periodic orbit $\beta$, we have
$$
0=\langle [\omega],[\beta]\rangle=\int_\beta\omega(\dot\beta(t))dt=\int_\beta\phi_1(\beta(t))dt= \int_{H_*\beta}\phi_2(H_*\beta(t))dt,
$$
where $H_*\beta$ is the periodic orbit of $X_2^t$ which corresponds to $\beta$ under the orbit equivalence $H$. Since $H$ is homotopic to identity, we have that integrals of $\phi_2$ over every homologically trivial orbit of $X_2^t$ vanish. Then we can apply~\cite[Theorems~2.9~and~3.5]{GRH2} to conclude that $\phi_2$ is an abelian coboundary over $X_2^t$ which completes the proof in this case.

It remains to consider the case when $c_0>0$ (if $c_0<0$ we can pass to the matching pair $(-\phi_1,-\phi_2)$ which makes $c_0$ positive). Let $\eta$ be a closed 1-form which represents the cohomology class $H_*[\omega]$.

\begin{lemma} The pair of functions $(\bar\phi_1,\bar\phi_2)\stackrel{\mathrm{def}}{=}(c_0, \phi_2-\eta(X_2))$ is a matching pair.
\end{lemma}

\begin{proof} Indeed, if $\beta$ is a periodic orbit of $X_1^t$ and $H_*\beta$ is the corresponding periodic orbit of $X_2^t$ then
\begin{multline*}
\int_\beta c_0 dt=\int_\beta(\phi_1(\beta(t)-\omega(\dot\beta(t)))dt
=\int_\beta\phi_1(\beta(t))dt-\langle[\omega],[\beta]\rangle\\
=\int_{H_*\beta}\phi_2(H_*\beta(t))dt-\langle[H_*\omega],[H_*\beta]\rangle
=\int_{H_*\beta}\phi_2(H_*\beta(t))dt-\langle[\eta],[H_*\beta]\rangle\\
= \int_{H_*\beta}\phi_2(H_*\beta(t))-\eta(X_2(H_*\beta(t)))dt
= \int_{H_*\beta}\bar\phi_2(H_*\beta(t))dt.
\end{multline*}
\end{proof}

\begin{lemma} The function $\bar \phi_2$ is cohomologous to a positive function.
\end{lemma}

\begin{proof}
Denote by $\mu_\beta$ the invariant probability measure supported on a periodic orbit $\beta$ of $X_1^t$ and, similarly, by $\mu_{H_*\beta}$ the invariant probability measure supported on $H_*\beta$. Then, using the preceding lemma, we have
\begin{multline*}
\int\bar\phi_2d\mu_{H_*\beta}=\frac{1}{\mathrm{per}(H_*\beta)}\int_{H_*\beta}\bar\phi_2(H_*\beta(t))dt
\\=
\frac{1}{\mathrm{per}(H_*\beta)}\int_{\beta}\bar\phi_1(\beta(t))dt
= \frac{\mathrm{per}(\beta)} {\mathrm{per}(H_*\beta)}\int \phi_1d\mu_\beta
=\frac{c_0 \mathrm{per}(\beta)} {\mathrm{per}(H_*\beta)}
\end{multline*}
The ratio of periods is uniformly bounded from below since derivative along the flow lines of the orbit equivalence $H$ is uniformly bounded from below. Hence we have a constant $c>0$ such that
$$
\int\bar\phi_2d\mu\ge c
$$
for any invariant probability measure $\mu$ supported on a periodic orbit. Since for Anosov flows such measures are dense in the space of ergodic measures we also have this bound for all ergodic probability measures of $X_2^t$.

With this inequality at hand we claim that $\bar\phi_2^T$ given by
$$
\bar\phi_2^T(x)=\frac 1T\int_0^T\bar\phi_2(X_2^t(x))dt
$$
is the posited function provided that $T$ is chosen to be sufficiently large. The fact that $\bar\phi_2^T$ is cohomologous to $\bar\phi_2$ follows from the Livshits Theorem and vanishing of the integrals of $\bar\phi_2^T-\bar\phi_2$ over periodic orbits as can be easily verified by a calculation.

The fact that $\bar\phi_2^T$ is positive for all sufficiently large $T$ is also standard and we only sketch the proof leaving the details to an interested reader. The standard approach is to argue reductio ad absurdum and to assume that $\bar\phi_2^{T_k}(x_k)\le 0$ for a sequence of points $\{x_k\}$ and a sequence of times $T_k\to\infty$. Then, using the Cantor diagonal argument one obtains an invariant probability measure $\mu$ such that $\int\bar\phi_2d\mu\le 0$ contradicting the established bound. This proof is analogous to the proof of uniform converegence of ergodic averages for uniquely ergodic systems~\cite[Proposition~4.1.13]{KH}.
\end{proof}

Using preceding lemma we replace $\bar\phi_2$ with a positive function cohomologous to it, and we still denote it by $\bar\phi_2$.
To summarize, we have now a matching pair $(\bar\phi_1,\bar\phi_2)$, where $\bar\phi_1=c_0>0$ and $\bar\phi_2>0$. This allows us to consider reparametrized flows $Y_i^t$ with generators given by $Y_i=\frac{1}{\bar\phi_i}X_i$, $i=1,2$. The periods of periodic orbits of $Y_i^t$ are calculated by integrating $\bar\phi_i$ over the periodic orbits of $X_i^t$. Hence, the matching assumption
$$
\int_\beta \bar\phi_1(\beta(t))dt= \int_{H_*\beta} \bar\phi_1(H_*\beta(t))dt
$$
says that the periods of matching orbits of $Y_1^t$ and $Y_2^t$ are equal. Hence, by \cite[Theorem 19.2.9]{KH}, we can improve the orbit equivalence $H$ to a conjugacy $\bar H$, $\bar H \circ Y_1^t = Y_2^t \circ\bar H$. Then Theorem~\ref{thm2} applies to $Y_1^t$ and $Y_2^t$. Since suspension flows do not admit contact reparametrizations by the proof of Lemma~\ref{lemma_contact}, Theorem~\ref{thm2} yields smoothness of $\bar H$, which is the posited smooth orbit equivalence of $X_1^t$ and $X_2^t$.

\section{Proof of Corollaries~\ref{cor5},~\ref{cor3} and~\ref{cor4}}

\begin{proof}[Proof of Corollary~\ref{cor5}]

Let $X_i=a_iY_i$ be the vector fields which generate contact reparametrizations $X_i^t$, $i=1,2$. By the change of variable formula the main matching assumption gives 
 $$
 \int_\beta \frac{\phi_1(\beta(t))}{a_1(\beta(t))}dt=\int_{H_*\beta}\frac{\phi_2(H_*\beta(t))}{a_2(H_*\beta(t))}dt
 $$
 for every periodic orbit $\beta$ of $X_1^t$ and corresponding periodic orbit $H_*\beta$ for $X_2^t$. Then we can apply Theorem~\ref{thm3} to $X_1^t$ and $X_2^t$ which yields the dichotomy. If $X_1^t$ is $C^{r_*}$-smoothly orbit equivalent to $X_2^t$ then the same orbit equivalence is a $C^{r_*}$ orbit equivalence between $Y_1^t$ and $Y_2^t$. Otherwise we have that $\phi_i/a_i$ are abelian coboundaries over $X_i^t$, $i=1,2$. This means that there exists closed 1-forms $\omega_i$ such that 
 $\phi_i/a_i=\omega_i(X_i)$. Hence $\phi_i=a_i\omega(X_i)=\omega(Y_i)$ and we have that $\phi_i$ are abelian coboundaries over $Y_i^t$, $i=1,2$.
\end{proof}

Now we prove Corollaries~\ref{cor3} and~\ref{cor4}.

\begin{proof} Denote by $X_i^t$ the geodesic flows of $g_i$ and by $X_i$ their generating vector fields, $i=1,2$. These flows are contact Anosov flows. It is also well known that $X_1^t$ is orbit equivalent to $X_2^t$ via an orbit equivalence $H$ which is homotopic to $id_{T^1S}$.  Since we have assumed that $\phi_1$ is not an abelian coboundary over $X_1^t$ we have that $H$ is smooth.

Now existence of homothety follows from Otal's proof of marked length spectrum rigidity~\cite{otal}. The key step in Otal's proof is to show that conjugacy of geodesic flows sends the Liouville measure invariant under the first flow to the Liouville measure invariant under the second flow. In fact, what is important is matching of Liouville currents on the spaces of geodesic on the universal covers $(\tilde S, \tilde g_1)$ and $(\tilde S,\tilde g_2)$. While we don't have a conjugacy and the orbit equivalence is ambiguous in the flow direction it does induce a canonical map $H_*$ on the space of currents. Hence, if $m_2$ is the Liouville current for $(\tilde S,\tilde g_2)$ then $H^*m_2$ is an invariant current for $(\tilde S, \tilde g_1)$. Further, since $H$ is smooth $H^*m_2$ is absolutely continuous and, hence, by ergodicity, is proportional to the Liouville current $m_1$ for $(\tilde S, \tilde g_1)$. We have $H^*m_2=cm_1$. Hence, after replacing $g_1$ with $c^2g_1$ we have that $H$ matches the Liouville currents. From this Otal's proof~\cite[Section~2]{otal} gives the posited isometry.

Alternatively this last step can be done by directly citing the Croke-Otal theorem and using our recent work on rigidity of contact flows. Because $H$ is smooth we can apply~\cite[Theorem 7.1]{GRH2} which says that $X_1^t$ is conjugate to a reparametrization of $X_2^t$ given by the vector field $\frac{1}{c+\omega(X)}X_2$. Further, recall that geodesic flows are contact and have zero Sharp's minimizers~\cite[Section~6]{GRH2}. Since $H^*(0)=0$ we, in fact, have that $\omega=0$~\cite[Theorem 7.1]{GRH2}. Hence $X_1^t$ is conjugate to a constant rescaling of $X_2^t$. Therefore the Otal and Croke theorem applies to these flows and yields an isometry.

To check the additional claim of Corollary~\ref{cor4} we will use the fact that the isometry $f\colon(S,c^2g_1)\to (S, g_2)$ can be chosen to be homotopic to $id_S$ (this is a part of the conclusion of the Otal and Croke theorem). This implies that $f(\gamma(g_1))=\gamma(g_2)$  and $f'=\frac1c$ when restricted to the geodesic.
\begin{multline*}
0=\int_{\gamma(g_1)}\phi_1(\dot\gamma(g_1)(t))dt-\int_{\gamma(g_2)}\phi_2(\dot\gamma(g_2)(t))dt\\
=\int_{\gamma(g_1)}\phi_1(\dot\gamma(g_1)(t))dt-\int_{\gamma(g_1)}\frac1c\phi_2(Df(\dot\gamma(g_1)(t)))dt\\
=\frac1c\int_{\gamma(g_1)}(c\phi_1-\phi_2\circ Df)(\dot\gamma(g_1)(t))dt.
\end{multline*}

By Livshits theorem we conclude that $c\phi_1-\phi_2\circ Df$ is a coboundary
$$
c\phi_1-\phi_2\circ Df=X_1u
$$
for some $u\colon T^1S\to\R$. We also have that $c\phi_1-\phi_2\circ Df=(c\psi_1-\psi_2\circ f)\circ \pi$, which means that this function is, in fact, a function on $S$. Then we can apply a result of Croke-Sharafutdinov~\cite[Corollary 1.4]{CS} to conclude that $c\phi_1-\phi_2\circ Df=c\psi_1-\psi_2\circ f=0$. Alternatively, using Fourier decomposition of $L^2(T^1S, vol)$ one can derive the same conclusion from an earlier result of Guillemin-Kazhdan~\cite[Theorem~3.6]{GK}.
\end{proof}

\appendix

\section{Proof of Theorem~\ref{thm_livshits}}


We begin by introducing some notation. Given a H\"older continuous cocycle $B\colon M\times\R\to\R$ and a periodic orbit $\gamma$ we will write $B(\gamma)$ for the value $B(p,|\gamma|)$, where $p\in\gamma$ and $|\gamma|$ is the smallest period of $p$, $X^{|\gamma|}(p)=p$. Also denote by $\delta_\gamma$ the invariant measure supported on $\gamma$ of total mass $|\gamma|$. Let
$$
\mu_T(B)=\frac{1}{\sum_{\gamma\in\cP_T}|\gamma|e^{B(\gamma)}} \sum_{\gamma\in\cP_T} e^{B(\gamma)}\delta_{\gamma},
$$
where $\cP_T$ is the set periodic orbits of length $\le T$. Bowen formula says that the probability measures $\mu_T(B)$ converge in weak$^*$ topology to the equilibrium state $\mu_B$ of the cocycle $B$ as $T\to\infty$~\cite{B, PP}.

Given a vector $\bar\alpha=(\alpha_1,\alpha_2,\ldots \alpha_N)\in\R^N$ let
$$
\mu(\bar\alpha)=\mu\left(\sum_{i=1}^N\alpha_ia_i\right)
$$
and let
$$
\mu_T(\bar\alpha)=\mu_T\left(\sum_{i=1}^N\alpha_ia_i\right).
$$
Denote by $A_i$ be a set of periodic orbits $\gamma$ such that $a_i(\gamma)=0$. If several cocycles vanish on $\gamma$ then we assign $\gamma$ to only one of the sets so that the sets $A_i$ are all mutually disjoint. Then according to the main assumption $\cP_T=\cup_{i=1}^N\cP_T\cap A_i$. Then we can also define approximating probability measures supported on $A_i$  as follows
$$
\mu_T^{A_i}(B)=\frac{1}{\sum_{\gamma\in\cP_T\cap A_i}|\gamma|e^{B(\gamma)}} \sum_{\gamma\in\cP_T\cap A_i} e^{B(\gamma)}\delta_{\gamma}
$$
and
$$
\mu_T^{A_i}(\bar\alpha)=\mu^{A_i}_T\left(\sum_{i=1}^N\alpha_ia_i\right).
$$
Because $a_i$ vanishes on $A_i$ we have that $\mu_T^{A_i}(\bar\alpha)$ is constant in the $i$-th variable $\alpha_i$.

Obviously, we have the following formula
$$
\mu_T(\bar\alpha)=\sum_{i=1}^Ns_T^i\mu_T^{A_i}(\bar\alpha)
$$
with
$$
s_T^i=\frac{\sum_{\gamma\in\cP_T\cap A_i}|\gamma|e^{\sum_i\alpha_ia_i(\gamma)}}{\sum_{\gamma\in\cP_T}|\gamma|e^{\sum_i\alpha_ia_i(\gamma)}}.
$$
Clearly $\sum_is_T^i=1$. By compactness of $[0,1]$ and the space of probability measures in the weak$^*$ topology we can choose a sequence $T_k\to\infty$, $k\to\infty$ such that for all $i$ we have $s_{T_k}^i\to s^i$ and $\mu_{T_k}^{A_i}(\bar\alpha)\to\mu^{A_i}(\bar\alpha)$ as $k\to\infty$. By passing to the limit in the above formula we obtain a decomposition of the equilibrium state as a convex combination
$$
\mu(\bar\alpha)=\sum_{i=1}^Ns^i\mu^{A_i}(\bar\alpha).
$$
Since equilibrium state is ergodic and all $\mu^{A_i}(\bar\alpha)$ are invariant probability measures, if all of them are distinct we immediately conclude that all but one of the $s^i$ vanish.
Therefore we have $\mu(\bar\alpha)=\mu^{A_i}(\bar\alpha)$ for some $i$.

If some of the measures $\mu^{A_i}(\alpha)$ coincide, say $\mu^{A_i}(\bar\alpha)=\mu^{A_j}(\bar\alpha)$, then (even though coefficients $s^i$ and $s^j$ may be non-trivial), we still have by ergodicity that $\mu(\bar\alpha)=\mu^{A_i}(\bar\alpha)=\mu^{A_j}(\bar\alpha)$.  

Now we consider all $N$-tuples $\bar \alpha$ from the set $\{1,2,\ldots N+1\}^N$. Since this is a finite set we can find a sequence $T_k$ such that all coefficients and all sequences of measures $\mu_{T_k}^{A_i}(\bar\alpha)$, $\bar\alpha\in\{1,2,\ldots N+1\}^N$ converge to $\mu^{A_i}(\alpha)$ as $k\to\infty$. 

By the preceding discussion for every $\alpha=\{1,2,\ldots N+1\}^N$ we have $\mu(\alpha)=\mu^{A_i}(\alpha)$ for some $i\in[1,N]$. In the case when $\mu(\alpha)=\mu^{A_i}$ for several $i$-s we pick just one of them so that the a function $I\colon\bar\alpha\to \mu^{A_i}(\alpha)$ is defined. The domain of this function is $\{1,2,\ldots N+1\}^N$ and the range is the space of all measures $\mu_T^{A_i}$. This space can be combinatorially parametrized as
$$
\bigcup_{i=1}^N C_i\stackrel{\textup{def}}{=}\bigcup_{i=1}^N \{(i,\bar\alpha): \alpha\in \{1,2,\ldots N+1\}^N, \alpha_i=1\}
$$
(In fact, from the defintion of $I$, this is the rigorous definition of the range of $I$ because even when $\mu^{A_i}(\bar\alpha)=\mu^{A_k}(\bar\beta)$ we still consider them as different points in the range if $i\neq k$ or $\bar\alpha\neq\bar\beta$, or both.)
Recall that  $\mu_T^{A_i}(\bar\alpha)$ is constant in $i$-th coordinate and we also have, by taking the limit, that $\mu^{A_i}(\bar\alpha)$ does not depend on the $i$-th coordinate. Hence we can set $\alpha_i=1$ in the above definition of the set $C_i$.

\begin{lemma} There exist $\bar\alpha$ and $\bar\beta\in \{1,2,\ldots N+1\}^N$ and $i\in[1, N]$ such that
\begin{enumerate}
\item $I(\bar\alpha)=I(\bar\beta)\in C_i$;
\item $\alpha_i\neq\beta_i$;
\item $\alpha_j=\beta_j$ for all $j\neq i$.
\end{enumerate}
\end{lemma}

Using this elementary lemma we can finish the proof. For $\bar\alpha$ and $\bar\beta$ given by the lemma we have $\mu(\bar\alpha)=\mu(\bar\beta)=\mu^{A_i}(\bar\alpha)$. The difference of cocycles is
$$
(\alpha_1a_1+\alpha_2a_2+\ldots +\alpha_Na_N)-(\beta_1a_1+\beta_2a_2+\ldots +\beta_Na_N)=(\alpha_i-\beta_i)a_i,
$$
Since equilibrium states for these cocycles are equal we can use a theorem of Bowen~\cite[Theorem~1.28]{B} to conclude that the difference $(\alpha_i-\beta_i)a_i$ is cohomologous to zero. Since $\alpha_i\neq\beta_i$ we obtain that $a_i$ is a coboundary as posited by the Alternate Livshits Theorem.

It remains to prove the lemma.
\begin{proof} From the definition of $I$, we have 
$$
I(\alpha_1,\alpha_2,\ldots \alpha_N)=(i,\alpha_1,\ldots \alpha_{i-1},1,\alpha_{i+1},\ldots \alpha_N)
$$
 Therefore the lemma simply says that $I$ has at least two preimages of some element. Hence we can use a counting argument. The domain of $I$ has $(N+1)^N$ element. Each $C_i$ has $(N+1)^{N-1}$ elements. So the cardinality of the range is $N(N+1)^{N-1}$. Since 
$(N+1)^N>N(N+1)^{N-1}$, the lemma follows from the pigeonhole principle.
\end{proof}

\end{document}